\documentclass[11pt]{article}

\usepackage[letterpaper,margin=1in]{geometry}
\usepackage{amsmath,amssymb,amsthm,mathtools}
\usepackage{microtype}
\usepackage{xcolor}
\usepackage[colorlinks=true,linkcolor=blue!60!black,citecolor=blue,urlcolor=blue!60!black]{hyperref}
\hypersetup{
  pdftitle={Fractional Triangle Decompositions at Partite Minimum Degree 4n/5},
  pdfauthor={Hengzhi He and Guang Cheng}
}

\newtheorem{theorem}{Theorem}[section]
\newtheorem{lemma}[theorem]{Lemma}
\newtheorem{corollary}[theorem]{Corollary}
\theoremstyle{definition}

\theoremstyle{remark}
\newtheorem{remark}[theorem]{Remark}

\newcommand{\cT}{\mathcal T}
\newcommand{\1}{\mathbf 1}
\newcommand{\whdelta}{\widehat{\delta}}

\allowdisplaybreaks

\title{Fractional Triangle Decompositions at Partite Minimum Degree \(4n/5\)}
\author{Hengzhi He\thanks{Department of Statistics and Data Science, University of California, Los Angeles.}
\qquad Guang Cheng\footnotemark[1]}
\date{August 2026}

\begin{document}

\maketitle

\begin{abstract}
We prove that every triangle-divisible balanced tripartite graph whose vertex
classes have size \(n\) and whose partite minimum degree is at least \(4n/5\)
has a fractional triangle decomposition.  Combined with the multipartite
decomposition theorem of Barber, K\"uhn, Lo, Osthus and Taylor, this implies
that, for every fixed \(\varepsilon<1/5\) and all sufficiently large \(n\),
every \(\varepsilon\)-dense partial Latin square of order \(n\) is completable.
The fractional theorem is finite and exact.  Its proof uses a minimum-weight
perfect matching to normalize an arbitrary Farkas dual weighting, followed by
an explicit direct-and-two-step routing whose off-matching congestion is at
most one.
\end{abstract}

\section{Introduction}

A partial Latin square of order \(n\) is called \(\varepsilon\)-dense if every
row and every column contains at most \(\varepsilon n\) filled cells and every
symbol occurs at most \(\varepsilon n\) times.  Daykin and H\"aggkvist
conjectured that every \(1/4\)-dense partial Latin square is
completable~\cite{DaykinHaggkvist}.  Wanless showed that the constant \(1/4\)
cannot be increased~\cite{Wanless}.  The conjecture therefore asks for the
sharp universal density threshold.
Thus $1/4$ is the largest possible constant in the conjectured universal
statement, although attainability at $1/4$ remains open.

The standard translation identifies a partial Latin square with a collection
of edge-disjoint triangles in \(K_{n,n,n}\).  Its completion is equivalent to
a triangle decomposition of the graph left after deleting those triangles.
Fractional triangle decompositions form the natural linear relaxation.  The
work of Bowditch and Dukes gave the asymptotic completion constant
\(1/25\)~\cite{BowditchDukes}.  Recently, Yu and Feng improved this to
\(2/25\)~\cite{YuFeng}.  Their result, like the one below, proceeds through a
fractional decomposition theorem for all triangle-divisible balanced
tripartite graphs, a class strictly larger than complements of partial Latin
squares.

For a tripartite graph \(G\) with vertex classes \(X,Y,Z\), define its
\emph{partite minimum degree} by
\[
 \whdelta(G)=
 \min\{d_G(v,W):U,W\in\{X,Y,Z\},\ U\neq W,\ v\in U\}.
\]
We prove the following finite statement.

\begin{theorem}\label{thm:main}
Let \(G\) be a triangle-divisible tripartite graph with vertex classes
\(X,Y,Z\), each of size \(n\).  If
\[
    \whdelta(G)\geq \frac{4n}{5},
\]
then \(G\) has a fractional triangle decomposition.
\end{theorem}

\begin{corollary}\label{cor:latin}
For every fixed \(0<\varepsilon<1/5\), every \(\varepsilon\)-dense partial Latin
square of sufficiently large order is completable.
\end{corollary}

The proof is dual.  A minimum-weight perfect matching in one bipartite pair
puts an arbitrary Farkas dual weighting into a gauge in which all weights on
that pair are nonnegative and the matching edges have weight zero.  For every
vertex in the third class, we then construct a fractional link matching.
Available source--sink arcs receive direct mass \(1/q\), where \(q\) is the
maximum degree deficiency, and the residual mass is routed in two steps.  A
refined one-role estimate and an exact three-role count show that the total
load on every edge outside the distinguished matching is at most one.  All
estimates are deterministic and exact.

Section~\ref{sec:dual} gives the dual reduction.  Section~\ref{sec:routing}
constructs and audits the routing, and Section~\ref{sec:completion} deduces
Corollary~\ref{cor:latin}.

\section{Dual reduction}\label{sec:dual}

Let \(G\) be a tripartite graph with parts \(X,Y,Z\).  It is
\emph{triangle-divisible} if
\[
\begin{aligned}
 d_G(x,Y)&=d_G(x,Z) &&(x\in X),\\
 d_G(y,X)&=d_G(y,Z) &&(y\in Y),\\
 d_G(z,X)&=d_G(z,Y) &&(z\in Z).
\end{aligned}
\]
A \emph{fractional triangle decomposition} of \(G\) is a family
\((w_T)_{T\in\cT(G)}\) such that \(w_T\geq0\) and
\[
    \sum_{\substack{T\in\cT(G)\\e\in E(T)}}w_T=1
    \qquad\text{for every }e\in E(G).
\]

Write \(a_{xy}\), \(b_{yz}\), and \(c_{zx}\) for real weights on the edges of
\(G[X,Y]\), \(G[Y,Z]\), and \(G[Z,X]\), respectively, and put
\[
    A=\sum_{xy\in E(G[X,Y])}a_{xy},\qquad
    B=\sum_{yz\in E(G[Y,Z])}b_{yz},\qquad
    C=\sum_{zx\in E(G[Z,X])}c_{zx}.
\]

\begin{lemma}[Dual criterion]\label{lem:farkas}
The graph \(G\) has a fractional triangle decomposition if and only if every
edge weighting satisfying
\begin{equation}\label{eq:triangle-dual}
    a_{xy}+b_{yz}+c_{zx}\geq0
    \qquad\text{for every }xyz\in\cT(G)
\end{equation}
also satisfies \(A+B+C\geq0\).
\end{lemma}

\begin{proof}
Let \(W\) be the edge--triangle incidence matrix of \(G\), with rows indexed by
edges and columns indexed by triangles.  A fractional triangle decomposition
is a vector \(w\geq0\) satisfying \(Ww=\1\).  Farkas' alternative says that
such a vector exists if and only if
\[
    y^{\mathsf T}W\geq0\quad\Longrightarrow\quad
    y^{\mathsf T}\1\geq0
\]
for every real vector \(y\) indexed by the edges.  The coordinates of
\(y^{\mathsf T}W\) are precisely the three-edge sums in
\eqref{eq:triangle-dual}, while \(y^{\mathsf T}\1=A+B+C\).
\end{proof}

We next normalize an arbitrary weighting satisfying
\eqref{eq:triangle-dual}.  The normalization is the only point at which
triangle-divisibility is used.

\begin{lemma}[Matching gauge]\label{lem:gauge}
Suppose that \(G\) is triangle-divisible, the three vertex classes have size
\(n\), and \(\whdelta(G)\geq n/2\).  For every weighting satisfying
\eqref{eq:triangle-dual}, there is a perfect matching
\[
    M=\{x_vy_v:v\in V\}\subseteq E(G[X,Y]),\qquad |V|=n,
\]
and an equivalent weighting, with the same triangle sums and the same total
sum \(A+B+C\), such that
\begin{equation}\label{eq:gauge}
    a_{xy}\geq0\quad(xy\in E(G[X,Y])),
    \qquad a_{x_vy_v}=0\quad(v\in V).
\end{equation}
\end{lemma}

\begin{proof}
First, \(G[X,Y]\) has a perfect matching.  
Indeed, Hall's condition is immediate when $S=\varnothing$. If
$S\subseteq X$ and $1\leq |S|\leq n/2$, choose $x\in S$. Then
\[
    |N(S)|\geq d_G(x,Y)\geq n/2\geq |S|.
\]
If \(|S|>n/2\) and \(N(S)\neq Y\), then a vertex
of \(Y\setminus N(S)\) has degree at most \(n-|S|<n/2\), a contradiction.
Hall's theorem applies.

Among all perfect matchings in \(G[X,Y]\), choose one, \(M\), minimizing the
sum of its \(a\)-weights.  The linear programming dual of the assignment
problem gives numbers \(\pi_x\) for \(x\in X\) and \(\rho_y\) for \(y\in Y\) such
that
\begin{equation}\label{eq:assignment-dual}
    \pi_x+\rho_y\leq a_{xy}\quad(xy\in E(G[X,Y])),
    \qquad \pi_x+\rho_y=a_{xy}\quad(xy\in M).
\end{equation}
For completeness, this follows by applying finite-dimensional linear
programming duality to
\[
 \min\left\{\sum_{xy\in E(G[X,Y])}a_{xy}m_{xy}:
 \sum_y m_{xy}=1,\ \sum_xm_{xy}=1,\ m_{xy}\geq0\right\};
\]
the bipartite matching polytope is integral, so its optimum is the chosen
matching, and complementary slackness gives equality on \(M\).

Define
\[
 a'_{xy}=a_{xy}-\pi_x-\rho_y,\qquad
 b'_{yz}=b_{yz}+\rho_y,\qquad
 c'_{zx}=c_{zx}+\pi_x.
\]
Every triangle sum is unchanged, and \eqref{eq:assignment-dual} gives
\eqref{eq:gauge} for \(a'\).  Moreover,
\[
\begin{split}
 A'+B'+C'-(A+B+C)
 ={}&\sum_{x\in X}\pi_x\bigl(d_G(x,Z)-d_G(x,Y)\bigr)\\
 &+\sum_{y\in Y}\rho_y\bigl(d_G(y,Z)-d_G(y,X)\bigr)=0,
\end{split}
\]
by triangle-divisibility.  Dropping the primes proves the lemma.
\end{proof}

Fix the indexing supplied by \(M\).  Define a directed graph \(D_M\) on \(V\)
by
\begin{equation}\label{eq:host}
    u\longrightarrow v\quad\Longleftrightarrow\quad x_uy_v\in E(G).
\end{equation}
Every loop is present.  If \(\whdelta(G)\geq n-q\), then every vertex of
\(D_M\) has at most \(q\) missing out-neighbors and at most \(q\) missing
in-neighbors.

For \(z\in Z\), set
\[
 I_z=\{v\in V:x_vz\notin E(G)\},\qquad
 K_z=\{v\in V:y_vz\notin E(G)\}.
\]
Triangle-divisibility at \(z\) gives \(|I_z|=|K_z|\leq q\).  A
\emph{fractional link matching at \(z\)} is a nonnegative matrix
\[
 \lambda^z=(\lambda^z_{uv})_
 {u\in V\setminus I_z,\ v\in V\setminus K_z}
\]
such that
\begin{equation}\label{eq:link}
\begin{gathered}
 \lambda^z_{uv}=0\quad\text{unless }u\longrightarrow v\text{ in }D_M,\\
 \sum_{v\in V\setminus K_z}\lambda^z_{uv}=1
 \quad(u\in V\setminus I_z),\qquad
 \sum_{u\in V\setminus I_z}\lambda^z_{uv}=1
 \quad(v\in V\setminus K_z).
\end{gathered}
\end{equation}
We extend \(\lambda^z_{uv}\) by zero whenever \(u\in I_z\) or \(v\in K_z\).

\begin{lemma}[Congestion criterion]\label{lem:criterion}
Under the hypotheses and notation established above, suppose that for
every $z\in Z$ there is a fractional link matching $\lambda^z$ such that
\begin{equation}\label{eq:congestion}
    \sum_{z\in Z}\lambda^z_{uv}\leq1
    \qquad\text{for every distinct }u,v\in V.
\end{equation}
Then \(G\) has a fractional triangle decomposition.
\end{lemma}

\begin{proof}
Take an edge weighting satisfying \eqref{eq:triangle-dual}, and apply
Lemma~\ref{lem:gauge}.  For each \(z\), multiply the inequality for the
triangle \(x_uy_vz\) by \(\lambda^z_{uv}\) and sum over all \(u,v,z\).
The row and column equations in \eqref{eq:link} count every edge in
\(G[Z,X]\) and \(G[Y,Z]\) exactly once.  Hence
\[
 0\leq
 \sum_{\substack{u,v\in V\\u\to v}}
       a_{x_uy_v}\sum_{z\in Z}\lambda^z_{uv}+B+C.
\]
The terms with \(u=v\) vanish by \eqref{eq:gauge}.  For \(u\neq v\), the
coefficient is at most one by \eqref{eq:congestion}, and
\(a_{x_uy_v}\geq0\).  Consequently the right-hand side is at most
\(A+B+C\).  Lemma~\ref{lem:farkas} now proves the assertion.
\end{proof}

\section{Adaptive routing}\label{sec:routing}

We construct the link matchings required by Lemma~\ref{lem:criterion}.  Put
\[
    q=n-\whdelta(G).
\]
Under the hypothesis of Theorem~\ref{thm:main}, \(q\leq n/5\).  The case
\(q=0\) is immediate, so assume \(q>0\).

Fix \(z\in Z\), suppress the subscript \(z\), and define
\[
 R=K\setminus I,\qquad S=I\setminus K,\qquad
 C=V\setminus(I\cup K).
\]
Then
\[
    |R|=|S|=:k,\qquad V\setminus I=R\cup C,\qquad
    V\setminus K=S\cup C.
\]
If \(k=0\), set \(\lambda^z_{cc}=1\) for \(c\in C\) and all other entries
equal to zero.  Assume henceforth that \(k>0\).  For \(r\in R\) and
\(s\in S\), let
\begin{equation}\label{eq:T}
    T(r,s)=\{c\in C:r\to c\text{ and }c\to s\}.
\end{equation}
Since \(|C|\geq n-2q\) and each of the two arc conditions excludes at most
\(q\) vertices,
\begin{equation}\label{eq:T-coarse}
    |T(r,s)|\geq n-4q>0.
\end{equation}

Reserve mass \(1/k\) for every ordered pair \((r,s)\in R\times S\).  If
\(r\to s\), route an amount $1/q$ directly and put
\[
    \beta_{rs}=\frac1k-\frac1q.
\]
If \(r\not\to s\), put \(\beta_{rs}=1/k\).  In both cases
\(\beta_{rs}\geq0\), because \(k\leq q\).  Distribute the residual mass
\(\beta_{rs}\) uniformly through \(T(r,s)\).  More precisely, initialize all
entries to zero, add
\[
\begin{array}{ll}
 1/q\text{ to }\lambda^z_{rs},&\text{if }r\to s,\\[2mm]
 \displaystyle\frac{\beta_{rs}}{|T(r,s)|}
 \text{ to each of }\lambda^z_{rc}\text{ and }\lambda^z_{cs},
 &\text{for every }c\in T(r,s),
\end{array}
\]
and finally set
\begin{equation}\label{eq:loop-fill}
  t(c)=\sum_{\substack{r\in R,\ s\in S\\c\in T(r,s)}}
       \frac{\beta_{rs}}{|T(r,s)|},
  \qquad \lambda^z_{cc}=1-t(c)\quad(c\in C).
\end{equation}

\begin{lemma}\label{lem:local}
The matrix \(\lambda^z\) defined above is a fractional link matching.
\end{lemma}

\begin{proof}
All used entries correspond to arcs of \(D_M\).  For \(r\in R\), the row
sum is
\[
 \sum_{s\in S}\left(\frac{\mathbf1_{\{r\to s\}}}{q}
      +\beta_{rs}\right)=\sum_{s\in S}\frac1k=1.
\]
The same computation gives column sum one for every \(s\in S\).  At a
vertex \(c\in C\), the total mass entering column \(c\) along the first
legs equals the total mass leaving row \(c\) along the second legs, and
both are \(t(c)\).  More explicitly, the first leg of a routed pair
\((r,s)\) contributes its path mass to an entry \(\lambda^z_{rc}\), and
the second leg contributes the same mass to \(\lambda^z_{cs}\).  Hence
column \(c\) and row \(c\) both have non-loop sum \(t(c)\), and
\eqref{eq:loop-fill} makes both of their full sums equal to one.

It remains only to prove \(t(c)\leq1\).  The total residual mass is at most
\[
    \sum_{r\in R,\ s\in S}\frac1k=k.
\]
By \eqref{eq:T-coarse},
\[
    t(c)\leq\frac{k}{n-4q}\leq\frac{q}{n-4q}\leq1,
\]
where the last inequality is equivalent to \(5q\leq n\).  Hence every
entry is nonnegative.
\end{proof}

The direct mass is essential for the global congestion estimate.  We first
record the sharper bound for each of the two-leg roles.

\begin{lemma}[Per-role bound]\label{lem:role}
For fixed \(z\) and \(r\in R\), every first-leg arc \(r\to c\), \(c\in C\),
receives total mass at most \(1/(n-3q)\).  Symmetrically, for fixed \(z\)
and \(s\in S\), every second-leg arc \(c\to s\), \(c\in C\), receives total
mass at most \(1/(n-3q)\).
\end{lemma}

\begin{proof}
Write \(h=|I|=|K|\).  Since \(|I\cup K|=h+k\), we have
\begin{equation}\label{eq:C-sharp}
    |C|=n-h-k\geq n-q-k.
\end{equation}
Fix \(r\in R\), and let
\[
    d=|\{s\in S:r\to s\}|.
\]
The vertex \(r\) is missing \(k-d\) out-arcs into \(S\), and it has at most
\(q\) missing out-arcs in total.  It therefore has at most \(q-k+d\)
missing out-arcs into \(C\).  For any fixed \(s\in S\), at most \(q\)
vertices of \(C\) fail the condition \(c\to s\).  By
\eqref{eq:C-sharp},
\begin{equation}\label{eq:T-sharp}
\begin{aligned}
    |T(r,s)|&\geq |C|-(q-k+d)-q\\
             &\geq n-3q-d.
\end{aligned}
\end{equation}
The total residual mass leaving \(r\) is
\[
 d\left(\frac1k-\frac1q\right)+(k-d)\frac1k
 =1-\frac dq.
\]
Consequently, for a fixed first-leg arc $r\to c$,
\[
\lambda^z_{rc}
=
\sum_{\substack{s\in S\\ c\in T(r,s)}}
\frac{\beta_{rs}}{|T(r,s)|}
\leq
\frac{\sum_{s\in S}\beta_{rs}}{n-3q-d}
=
\frac{1-d/q}{n-3q-d}
\leq
\frac{1}{n-3q}.
\]
For \(d=0\) this is equality.  For \(d>0\), the last inequality is
equivalent to \(n-3q\geq q\), which follows from \(q\leq n/5\).
For completeness, fix \(s\in S\) and put
\[
    d'=|\{r\in R:r\to s\}|.
\]
There are \(k-d'\) missing in-arcs to \(s\) from \(R\).  Since \(s\) has
at most \(q\) missing in-arcs altogether, at most \(q-k+d'\) of its
in-arcs from \(C\) are missing.  For each fixed \(r\in R\), at most \(q\)
vertices of \(C\) fail the condition \(r\to c\).  Therefore
\[
    |T(r,s)|\geq |C|-q-(q-k+d')\geq n-3q-d'.
\]
The total residual mass entering \(s\) equals \(1-d'/q\).  Thus, for a fixed second-leg arc $c\to s$,
\[
\lambda^z_{cs}
=
\sum_{\substack{r\in R\\ c\in T(r,s)}}
\frac{\beta_{rs}}{|T(r,s)|}
\leq
\frac{\sum_{r\in R}\beta_{rs}}{n-3q-d'}
=
\frac{1-d'/q}{n-3q-d'}
\leq
\frac{1}{n-3q},
\]
where the last inequality is proved exactly as above.  This establishes
the second-leg assertion.
\end{proof}

\begin{lemma}[Global congestion]\label{lem:global}
The family \((\lambda^z)_{z\in Z}\) satisfies \eqref{eq:congestion}.
\end{lemma}

\begin{proof}
Fix a non-loop arc \(u\to v\) of \(D_M\).  Partition the colors \(z\in Z\)
in which this arc can be used into the following three disjoint classes:
\[
\begin{array}{ccl}
 \mathcal D&=&\{z:u\in R_z,\ v\in S_z\},\\
 \mathcal F&=&\{z:u\in R_z,\ v\in C_z\},\\
 \mathcal H&=&\{z:u\in C_z,\ v\in S_z\}.
\end{array}
\]
Let \(D=|\mathcal D|\), \(F=|\mathcal F|\), and \(H=|\mathcal H|\).
For \(z\in\mathcal D\), the arc receives direct mass \(1/q\).  By
Lemma~\ref{lem:role}, in each color of \(\mathcal F\) or \(\mathcal H\) it
receives at most \(1/(n-3q)\).  Therefore
\begin{equation}\label{eq:global-load}
    \sum_{z\in Z}\lambda^z_{uv}
    \leq \frac Dq+\frac{F+H}{n-3q}.
\end{equation}
There are no further uses of \(u\to v\) in the construction: a non-loop
entry is either direct, a first leg, or a second leg, while the final fill
in \eqref{eq:loop-fill} uses loops only.

If \(z\in\mathcal D\cup\mathcal F\), then \(u\in R_z\subseteq K_z\), so
\(y_uz\notin E(G)\).  The vertex \(y_u\) has at most \(q\) non-neighbors
in \(Z\), and hence
\[
    D+F\leq q.
\]
Likewise, if \(z\in\mathcal D\cup\mathcal H\), then
\(v\in S_z\subseteq I_z\), so \(x_vz\notin E(G)\), and
\[
    D+H\leq q.
\]
Thus \(F+H\leq2q-2D\).  Since \(n-3q\geq2q\), \eqref{eq:global-load}
gives
\[
    \sum_{z\in Z}\lambda^z_{uv}
    \leq\frac Dq+\frac{2q-2D}{2q}=1.
\]
\end{proof}

\begin{proof}[Proof of Theorem~\ref{thm:main}]
Set \(q=n-\whdelta(G)\).  If \(q=0\), take the loop link matching at every
\(z\).  If \(q>0\), Lemma~\ref{lem:local} constructs a fractional link
matching for every \(z\), and Lemma~\ref{lem:global} bounds every
off-diagonal aggregate load by one.  Lemma~\ref{lem:criterion} completes
the proof.
\end{proof}

\begin{remark}\label{rem:barrier}
The constant \(1/5\) is the exact point at which this routing estimate
closes.  In the extremal count \(D=0\), an arc may have \(q\) first-leg
roles and \(q\) second-leg roles, each bounded by \(1/(n-3q)\); their total
is \(2q/(n-3q)\), which equals one when \(q=n/5\).  This is a barrier for
the present certificate, not an assertion that Theorem~\ref{thm:main}
fails below partite minimum degree \(4n/5\).
\end{remark}

\section{Completion of partial Latin squares}\label{sec:completion}

\begin{proof}[Proof of Corollary~\ref{cor:latin}]
Let \(P\) be an \(\varepsilon\)-dense partial Latin square of order \(n\).
Use vertex classes \(R,C,S\) for its rows, columns, and symbols.  Every
filled cell \((r,c)\) containing symbol \(s\) gives the triangle \(rcs\)
in \(K_{n,n,n}\), and the partial Latin property makes these prescribed
triangles edge-disjoint.  Delete all their edges and call the resulting
tripartite graph \(G_P\).

At a row vertex, the number of deleted row--column edges equals the number
of deleted row--symbol edges.  The analogous statement holds at every
column and every symbol.  Thus \(G_P\) is triangle-divisible.  Since \(P\)
is \(\varepsilon\)-dense,
\[
    \whdelta(G_P)\geq(1-\varepsilon)n.
\]

Let \(\widehat\delta^*_{K_3}\) denote the asymptotic fractional
decomposition threshold for triangle-divisible balanced tripartite graphs,
in the notation of Barber, K\"uhn, Lo, Osthus and
Taylor~\cite{BarberKuhnLoOsthusTaylor}.  Theorem~\ref{thm:main} gives
\[
    \widehat\delta^*_{K_3}\leq\frac45.
\]
Their Corollary~1.6 states that, for every fixed \(\eta>0\), every
sufficiently large triangle-divisible balanced tripartite graph \(H\) with
equal part sizes and
\[
    \whdelta(H)\geq
    \bigl(\widehat\delta^*_{K_3}+\eta\bigr)n
\]
has a triangle decomposition.

Choose $\eta>0$ so that $4/5+\eta<1-\varepsilon$. Then
\[
    \widehat{\delta}(G_P)
    \geq (1-\varepsilon)n
    >
    \left(\frac45+\eta\right)n
    \geq
    \left(\widehat{\delta}_{K_3}^{*}+\eta\right)n.
\]
Thus $G_P$ satisfies the required degree condition. For sufficiently
large $n$, it therefore has a triangle decomposition.
Adding back the prescribed triangles partitions all edges of
\(K_{n,n,n}\) into triangles, which is a Latin square extending \(P\).
\end{proof}

\end{document}